\documentclass{amsart}

\newcommand{\IR}{\mathbb R}
\newcommand{\w}{\omega}

\newcommand{\diam}{\mathrm{diam}}
\newcommand{\W}{\mathcal W}
\newcommand{\e}{\varepsilon}

\newtheorem{theorem}{Theorem}
\newtheorem{corollary}{Corollary}
\newtheorem{problem}{Problem}

\theoremstyle{definition}
\newtheorem{remark}{Remark}

\begin{document}

\title[Sequential closure of the set of continuous functions]{On the sequential closure of the set of continuous functions\\ in the space of separately continuous functions}
\author{Taras Banakh}
\email{t.o.banakh@gmail.com}
\address{Ivan Franko National University of Lviv (Ukraine) and Jan Kochanowski University in Kielce (Poland)}

\begin{abstract}
For separable metrizable spaces $X,Y$ and a metrizable topological group  $Z$ by $S(X\times Y,Z)$ we denote the space of all separately continuous functions $f:X\times Y\to Z$ endowed with the topology of layer-wise uniform convergence, generated by the subbase consisting of the sets  $[K_X\times K_Y,U]=\{f\in S(X\times Y,Z):f(K_X\times K_Y)\subset U\}$, where $U$ is an open subset of $Z$ and $K_X\subset X$, $K_Y\subset Y$ are compact sets one of which is a singleton. We prove that every separately continuous function $f:X\times Y\to Z$ with zero-dimensional image $f(X\times Y)$ is a limit of a sequence of jointly continuous functions in the topology of layer-wise uniform convergence.
\end{abstract}
\maketitle

In this paper we study the sequential closure $\bar C^s(X\times Y,Z)$ of the set $C(X\times Y,Z)$ of continuous functions in the space $S(X\times Y,Z)$ of separately continuous functions $f:X\times Y\to Z$ defined on the product of topological spaces $X,Y$ with values in a topological space $Z$. The space $S(X\times Y,Z)$ is endowed with the topology of layer-wise compact-open topology, generated by the subbase consisting of the sets
$$[K_X\times K_Y,U]=\{f\in S(X\times Y,Z):f(K_X\times K_Y)\subset U\},$$
where $U\subset Z$ is an open set and $K_X\subset X$, $K_Y\subset Y$ are compact subsets one of which is a singleton. Observe that a function sequence $(f_n)_{n\in\w}$ converges to a function $f_\infty:X\times Y\to Z$ in the space $S(X\times Y,Z)$ if and only if for any points $x\in X$, $y\in Y$ the function sequence $\big(f_n(x,\cdot)\big)_{n\in\w}$ converges to $f_\infty(x,\cdot)$ in the space $C(Y,Z)$ and the function  sequence  $\big(f_n(\cdot,y)\big)_{n\in\w}$ converges to $f_\infty(\cdot,y)$ in the function space $C(X,Z)$ (endowed with the compact-open topology).

In case of compact spaces $X,Y$ and the real line $Z=\IR$ the space $S(X\times Y,Z)$ was studied by H.A.~Voloshyn, V.K.~Maslyuchenko and O.V.~Maslyuchenko in the papers \cite{VM1}--\cite{VMM2}. One of the problems addressed in these papers was the problem of description of the sequential closure $\bar C^s(X\times Y,Z)$ of the set $C(X\times Y,Z)$ of jointly continuous functions in the space $S(X\times Y,Z)$. By the  {\em sequential closure} $\bar A^s$ of a set $A\subset X$ in a topological space $T$ we understand the set of all limit points $\lim_{n\to \infty}a_n$ of sequences $\{a_n\}_{n\in\w}\subset A$, convergent in the space $T$.
The Tietze-Urysohn Extension Theorem \cite[2.1.8]{En} implies that for any Tychonoff spaces $X,Y$ the set $C(X\times Y,\IR)$ is dense in $S(X\times Y,\IR)$, which motivates the following question.

\begin{problem}\label{prob1} Is $\bar C^s(X\times Y,\IR)=S(X\times Y,\IR)$ for any (zero-dimensional) compact metrizable spaces $X,Y$?
\end{problem}

In \cite{VM3} Voloshyn and Maslyuchenko proved that for metrizable compact spaces  $X,Y$ the set $\bar C^s(X\times Y,\IR)$ contains all separately continuous functions $f:X\times Y\to \IR$ whose discontinuity point set $D(f)\subset X\times Y$ has countable projection on $X$ or $Y$.

In this paper we shall prove that for separable metrizable spaces $X,Y$ and a metrizable topological group $Z$ the set $\bar C^s(X\times Y,Z)$ contains all separately continuous functions $f:X\times Y\to Z$ with zero-dimensional image $f(X\times Y)\subset Z$.

We start with the following useful fact.

\begin{theorem}\label{t1} For separable metrizable spaces $X,Y$ and a topological space $Z$ the set $\bar C^s(X\times Y,Z)$ contains all separately continuous functions $f:X\times Y\to Z$ with discrete image $f(X\times Y)\subset Z$.
\end{theorem}

\begin{proof} Fix a separately continuous function $f:X\times Y\to Z$ with discrete image $f(X\times Y)\subset Z$. We lose no generality assuming that $Z=f(X\times Y)$. The separate continuity of $f$ implies that for any countable dense subsets $D_X\subset X$ and $D_Y\subset Y$ the set $f(D_X\times D_Y)$ is dense in the discrete space $Z=f(X\times Y)$, which implies that $Z=f(D_X\times D_Y)$ is countable and hence can be written as the countable union $Z=\bigcup_{n\in\w}Z_n$ of a non-decreasing sequence $(Z_n)_{n\in\w}$ of finite sets.

The second-countable spaces $X,Y$ have countable bases $\{U_n\}_{n\in\w}$ and $\{V_n\}_{n\in\w}$ of the topology.

For any point $z\in Z$ and a number $k\in\w$ consider the sets
$$
\begin{aligned}
X(z,k)&=\{x\in X:\{x\}\times V_k\subset f^{-1}(z)\}=\{x\in X:\{x\}\times \overline{V}_k\subset f^{-1}(z)\},\\
Y(z,k)&=\{y\in Y:U_k\times \{y\}\subset f^{-1}(z)\}=\{y\in Y:\overline{U}_k\times \{y\}\subset f^{-1}(z)\}.
\end{aligned}
$$
We claim that the set $X(z,k)$ is closed in $X$. Indeed, if $x\in X\setminus X(z,k)$, then $f(x,v)\ne z$ for some $v\in V_k$. The separate continuity of $f$ yields a neighborhood $O_x\subset X$ of $x$ such that $f(O_x\times\{v\})\subset Z\setminus \{z\}$, which implies that $O_x\cap X(z,k)=\emptyset$. By analogy we can prove that the sets $Y(z,k)$ are closed in $Y$.

It follows that for any number $n\in\w$ the set $$XY(z,n)=\bigcup_{k\le n}(X(z,k)\times \overline{V}_k)\cup(\overline{U}_k\times Y(z,k))$$is closed in $X\times Y$ and is contained in the preimage $f^{-1}(z)$.
Consequently, the family of closed sets $\big(XY(z,n)\big)_{z\in Z}$ is disjoint. Using the normality of the product $X\times Y$, we can construct a continuous function $g_n:X\times Y\to Z$ such that $XY(z,n)\subset g_n^{-1}(z)$ for all $z\in Z_n$.

It remains to check that the sequence $(g_{n})_{n\in\w}$ converges to $f$ in the space $S(X\times Y,Z)$. Given a neighborhood $O_f\subset S(X\times Y,Z)$, it suffices to find a number $m\in \w$ such that $g_n\in O_f$ for all $n\ge m$. We lose no generality assuming that  $O_f$ is a subbasic neighborhood of the form $[K_X\times K_Y,W]=\{g\in S(X\times Y,Z):g(K_X\times K_Y)\subset W\}$ for some some set $W\subset Z$ and some compacts sets $K_X\subset X$, $K_Y\subset Y$, one of which is a singleton.
Since the set $f(K_X\times K_Y)\subset W\subset Z$ is compact and hence finite in the discrete space $Z$, we can replace $W$ by the image $f(K_X\times K_Y)$ and assume that $W=f(K_X\times K_Y)$ is a finite (open) set in $X$.

First we consider the case of singleton set $K_X=\{x\}$. Consider the continuous map $f_x:Y\to Z$, $f_x:y\mapsto f(x,y)$, and observe that the set $f_x^{-1}(W)$ is closed-and-open in $Y$ and contains the compact set $K_Y$. For every $z\in W$ find a finite subset $F_z\subset\w$ such that $K_Y\cap f_x^{-1}(z)\subset \bigcup_{k\in F_z}V_k\subset f_x^{-1}(z)$. Choose a number $m\in\w$ such that $W\subset Z_m$ and $\bigcup_{z\in W}F_z\subset \{0,\dots,m\}$. Then for every $n\ge m$ and $z\in W$ we get $\{x\}\times (K_Y\cap f_x^{-1}(z))\subset XY(z,n)\subset g_n^{-1}(z)$ and hence  $$K_X\times K_Y=\{x\}\times K_Y\subset \bigcup_{z\in W}\{x\}\times f_x^{-1}(z)\subset \bigcup_{z\in W}g_n^{-1}(z)=g_n^{-1}(W),$$which implies $g_n\in O_f=[K_X\times K_Y,W]$ for all $n\ge m$.

By analogy we can treat the case of singleton set $K_Y$.
\end{proof}

Recall that a topological space $Z$ is {\em zero-dimensional} if closed-and-open sets form a base of the topology of the space $Z$.

\begin{theorem}\label{t2} For separable metrizable spaces $X,Y$ and a metrizable topological group $G$, the set $\bar C^s(X\times Y,Z)$ contains all separately continuous functions $f:X\times Y\to G$ with zero-dimensional image $f(X\times Y)\subset G$.
\end{theorem}

\begin{proof} Fix any separately continuous function $f:X\times Y\to G$ with a zero-dimensional image $f(X\times Y)$. By Birkhof-Kakutani Theorem \cite[3.3.12]{AT}, the topology of the metrizable topological group $G$ is generated by a bounded left-invariant metric $d$. Multiplying $d$ by a small positive constant, we can assume that $\diam(G)=\sup\{d(x,y):x,y\in G\}\le \frac12$. For a positive real $\e$ denote by $B[\e]=\{g\in G:d(g,1_G)\le\e\}$ the closed $\e$-ball centered at the unit $1_G$ of the group $G$.
For every $n\in\w$ choose a maximal subset $E_n\subset B[\frac1{2^{n-1}}]$, which is {\em $\frac1{2^{n+1}}$-separated} in the sense that $d(x,y)\ge \frac1{2^{n+1}}$ for any distinct points $x,y\in E_n$. It is clear that the set $E_n$ is closed and discrete in $G$.

The separability of the spaces $X,Y$ and the separate continuity of $f:X\times Y\to G$ implies the separability of the image $f(X\times Y)$. Since the metrizable space $Z=f(X\times Y)$ is separable and zero-dimensional, it has covering dimension $\dim(Z)=0$, which allows us to choose a sequence $(\W_n)_{n\in\w}$ of disjoint open covers of the space $Z$ by sets of diameter $\le \frac1{2^{n+1}}$. Replacing each cover $\W_n$ by the finer cover $\{W\cap W':W\in\W_n,\;W'\in\W_{n+1}\}$, we can assume that for every $n\in\w$ the cover $\W_{n+1}$ refines the cover $\W_n$. Since $\diam(Z)\le\frac12$, we can put $\W_0=\{Z\}$.

Inductively we shall construct a sequence  $(r_n)_{n\in\w}$ of continuous maps $r_n:Z\to G$ such that $r_0(Z)=\{1_G\}$ and for every $n\in\w$ the following conditions are satisfied:
\begin{enumerate}
\item[$(1_n)$] for each set $W\in\W_n$ the image $r_n(W)$ is a singleton;
\item[$(2_n)$] $\sup_{z\in Z}d(r_n(z),z)\le 2^{-n}$;
\item[$(3_n)$] $r_{n}(z)\in r_{n-1}(z)\cdot E_{n+1}$ for every $z\in Z$.
\end{enumerate}
We start the inductive construction putting $r_0:Z\to\{1_G\}\subset G$. Assume that for some $n\in\w$ we have constructed a map $r_n:Z\to G$ satisfying the conditions $(1_n)$ and $(2_n)$. For every $W\in\W_n$ denote by $g_W$ the unique point of the set $r_n(W)$. Consider the subfamily $\W_{n+1}(W)=\{W'\in\W_{n+1}:W'\subset W\}$ and for every set $W'\in\W_{n+1}(W)$ choose a point $z_{W'}\in W'$. Observe that  $d(z_{W'},g_W)=d(z_{W'},r_n(z_{W'}))\le \frac1{2^{n}}$ and hence $g_W^{-1}z_{W'}\in B[\frac1{2^{n}}]$. By the maximality of the  $\frac1{2^{n+2}}$-separated set $E_{n+1}\subset B[\frac1{2^{n}}]$, there exists a point $\e_{W'}\in E_{n+2}$ such that $d(\e_{W'},g_W^{-1}z_{W'})<\frac1{2^{n+2}}$. Consider the point $g_{W'}=g_W\cdot \e_{W'}$ and observe that $$d(g_{W'},z_{W'})=d(g_W\cdot \e_{W'},z_{W'})=d(\e_{W'},g_W^{-1}z_{W'})\le \frac1{2^{n+2}}.$$ Define a function $r_{n+1}:Z\to G$ assigning to each point $z\in Z$ the point $g_{W'}$ where $W'\in\W_{n+1}$ is a unique set containing the point $z$. It is clear that the function $r_{n+1}$ is continuous and satisfies the condition $(1_{n+1})$. To check the condition $(2_{n+1})$, take any point $z\in Z$ and find a unique set  $W'\in\W_{n+1}$ containing the point $z$ and a unique set $W\in\W_{n}$ containing the set $W'$. Observe that
$$d(r_{n+1}(z),z)=d(g_{W'},z)\le d(g_{W'},z_{W'})+d(z_{W'},z)\le 2^{-n-2}+\diam(W')\le 2^{-n-2}+2^{-n-2}=2^{-n-1}.$$
To verify the condition $(3_{n+1})$, take any point $z\in Z$ and find unique sets $W\in\W_n$ and $W'\in\W_{n+1}$ containing the point $z$. Taking into account that the cover $\W_{n+1}$ refines the disjoint cover $\W_n$, we conclude that $W'\subset W$, which implies $r_{n+1}(z)=g_{W'}=g_W\cdot\e_{W'}=r_n(z)\cdot \e_{W'}\subset r_n(z)\cdot E_{n+1}$. This completes the inductive construction of the function sequence  $(r_n)_{n\in\w}$.

The continuity of $r_n$ and the separate continuity of $f$ implies that the composition $f_n=r_n\circ f:X\times Y\to G$ is separately continuous and $$\sup_{(x,y)\in X\times Y}d(f_n(x,y),f(x,y))\le \sup_{z\in Z}d(r_n(z),z)\le 2^{-n},$$ which means that the function sequence  $(f_n)_{n\in\w}$ converges uniformly to $f$.

The continuity of the group operations on $Z$ and the separate continuity of the functions $f_n,f_{n+1}$ imply that the function $g_n=f_n^{-1}f_{n+1}:X\times Y\to G$, $g_n: (x,y)\mapsto f_n^{-1}(x,y)\cdot f_{n+1}(x,y)$, is separately continuous. By the condition $(3_{n+1})$, we get $g_n(x,y)=f_n^{-1}(x,y)\cdot f_{n+1}(x,y)=r_n(f(x,y))^{-1}\cdot r_{n+1}(f(x,y))\in E_{n+1}$, which means that the image $g_n(X\times Y)$ is discrete.
By Theorem~\ref{t1}, for the separately continuous function $g_n:X\times Y\to E_{n+1}$ there is a sequence of continuous functions $(g_{n,m}:X\times Y\to E_{n+1})_{m\in\w}$, convergent to $g_n$ in the function space $S(X\times Y,Z)$.

For any numbers $n,m\in\w$ consider the continuous function $$f_{n,m}=g_{0,m}\cdot g_{1,m}\cdots g_{n,m}:X\times Y\to G.$$ Taking into account that $\lim_{m\to\infty}g_{n,m}=g_n$ for all $n$, we conclude that $\lim_{m\to \infty}f_{n,m}=g_0\cdots g_n=f_n$ for all $n\in\w$. It remains to check that $\lim_{n\to\infty}f_{n,n}=f$ in the space $S(X\times Y,G)$.

Fix any neighborhood $O_f$ of the function $f$ in the space $S(X\times Y)$. We lose no generality assuming that $O_f$ has subbasic form $[K_X\times K_Y,W]=\{g\in S(X\times Y,G):f(K_X\times K_Y)\subset W\}$ for some open set $W\subset G$ and some compact sets $K_X\subset X$, $K_Y\subset Y$ one of which is a singleton. Using the compactness of the set $f(K_X\times K_Y)\subset W$, find a number $l\in\w$ such that $f(K_X\times K_Y)\cdot B[\frac4{2^l}]\subset W$. By the convergence of the sequence $(f_{l,n})_{n\in\w}$ to the function $f_l$ in the space $S(X\times Y,G)$, there is a number $m\ge l$ such that  $\sup_{(x,y)\in K_X\times K_Y}d(f_{l,n}(x,y),f_l(x,y))\le 2^{-l}$ for all $n\ge m$. We claim that $\sup_{(x,y)\in K_X\times K_Y}d(f(x,y),f_{n,n}(x,y))\le 2^{-l}$ for all $n\ge m$. Take and number $n\ge m$ and a pair $z=(x,y)\in K_X\times K_Y$. Observe that $g_{l+1,n}(z)\cdots g_{n,n}(z)\in E_{l+2}\cdots E_{n+1}\subset B[\frac1{2^{l+1}}]\cdots B[\frac1{2^{n}}]\subset B[\frac1{2^{l}}]$. Then
$$
\begin{aligned}
d(f(z),f_{n,n}(z))&\le d(f(z),f_l(z))+d(f_l(z),f_{l,n}(z))+d(f_{l,n}(z),f_{n,n}(z))\le\\
&\le 2^{-l}+2^{-l}+d(g_{0,n}(z)\cdots g_{l,n}(z),g_{0,n}(z)\cdots g_{n,n}(z))=\\
&=2^{-l+1}+d(1_G,g_{l+1,n}(z)\cdots g_{n,n}(z))\le 2^{-l+1}+2^{-l}<2^{-l+2}
\end{aligned}
$$and hence $f_{n,n}(z)\in W$ according to the choice of $l$.
\end{proof}

Now we find conditions on a topological group $Z$ guaranteeing that the sequential closure $\bar C^s(X\times Y,Z)$ of the set $C(X\times Y,Z)$ of continuous functions in the space $S(X\times Y,Z)$ is closed in a topology of uniform convergence on $S(X\times Y,Z)$.

In fact, the space $S(X\times Y,Z)$ carries four natural topologies of uniform convergence induced by four canonical uniformities on the topological group $Z$. For describing these topologies, for any function $f\in S(X\times Y,Z)$ and any neighborhood $U\subset Z$ of the unit in the group $Z$ consider the sets
 $$
\begin{aligned}
B_l[f,U]&=\{g\in S(X\times Y,Z):\forall (x,y)\in X\times Y \;\;g(x,y)\in f(x,y)\cdot U\},\\
B_r[f,U]&=\{g\in S(X\times Y,Z):\forall (x,y)\in X\times Y \;\;g(x,y)\in  U\cdot f(x,y)\},\\
B_{lr}[f,U]&=B_l[f,U]\cap B_r[f,U]=\{g\in S(X\times Y,Z):\forall (x,y)\in X\times Y \;\;g(x,y)\in  U\cdot f(x,y)\cap f(x,y)\cdot U\},\\
B_{rl}[f,U]&=\{g\in S(X\times Y,Z):\forall (x,y)\in X\times Y \;\;g(x,y)\in  U\cdot f(x,y)\cdot U\}.
\end{aligned}
$$
We define a subset $W\subset S(X\times Y,Z)$ to be {\em $\tau_s$-open} for $s\in\{l,r,lr,rl\}$ if for any function $f\in W$ there is a neighborhood $U\subset Z$ of the unit $1_Z$ in $Z$ such that $B_s[f,U]\subset W$. The family $\tau_s$ of all $\tau_s$-open subsets of $S(X\times Y,Z)$ is a topology on $S(X\times Y,Z)$. It is clear that $\tau_{rl}\subset\tau_l\cap\tau_r$ and $\tau_l\cup\tau_r\subset\tau_{lr}$. We define a subset $F\subset S(X\times Y,Z)$ to be {\em $\tau_l\cap\tau_r$-closed} if its complement in $S(X\times Y,Z)$ belongs to the topology $\tau_l\cap\tau_r$.

For topological spaces $X,Y,Z$ by $S_0(X\times Y,Z)$ we denote the set of all separately continuous functions $f:X\times Y\to Z$ with zero-dimensional image $f(X\times Y)$. If $Z$ is a topological group, then by $\bar S^u_0(X\times Y,Z)$ we shall denote the closure of the set $S_0(X\times Y,Z)$ in the topology $\tau_l\cap\tau_r$ on the space $S(X\times Y,Z)$. By Theorem~\ref{t2}, $S_0(X\times Y,Z)\subset \bar C^s(X\times Y,Z)$ for any metrizable compact spaces $X,Y$ and any metrizable topological group $Z$. In the following theorem we give conditions on the group $Z$ guaranteeing that $\bar S^u_0(X\times Y,Z)\subset \bar C^s(X\times Y,Z)$.

We shall say that a topological space $Z$ is an {\em absolute extensor} for a topological space $X$ if any continuous map $f:A\to Z$ defined on a closed subset $A\subset X$ can be extended to a continuous map $\bar f:X\to Z$.

\begin{theorem}\label{t3} Let $X,Y$ be separable metrizable spaces and $Z$ be a metrizable topological group possessing a neighborhood base at the unit $1_Z$ consisting of absolute extensors for the space $X\times Y$. Then the set $\bar C^s(X\times Y,Z)$ is closed in the topology $\tau_{l}\cap\tau_r$ on $S(X\times Y,Z)$. Consequently,  $\bar S_0^u(X\times Y,Z)\subset \bar C^s(X\times Y,Z)$.
\end{theorem}

\begin{proof} It suffices to check that the set $\bar C^{s}(X\times Y,Z)$ is closed in the topologies $\tau_l$ and $\tau_r$ on the space $S(X\times Y,Z)$. First we prove that the set $\bar C^s(X\times Y,Z)$ is $\tau_l$-closed. Fix a left-invariant metric $d\le 1$ generating the topology of the metrizable topological group $Z$. Put $W_{-1}=W_0=Z$ and for every $n>0$ choose a neighborhood $W_n\subset \{z\in Z:d(1_Z,z)\le 2^{-n}\}$ of the unit $1_Z$, which is an absolute extensor for the space $X\times Y$. Replacing the sequence $(W_n)_{n\in\w}$ by a suitable subsequence, we can assume that $W_{n}^{-1}W_{n}W_{n}\subset W_{n-1}$ for all $n>0$.

Take any function  $f\in S(X\times Y,Z)$ that belongs to the $\tau_l$-closure of the set $\bar C^s(X\times Y,Z)$ in $S(X\times Y,Z)$. For every $n\in\w$ choose a function $f_n\in\bar C^s(X\times Y,Z)$ such that $f_n(x,y)\in f(x,y)\cdot W_{n+1}$ for all $(x,y)\in X\times Y$. Since $W_0=Z$ we can put $f_0:X\times Y\to\{1_Z\}$ be the constant function. For every $n\in\w$ consider the function $g_n=f_n^{-1}\cdot f_{n+1}:X\times Y\to Z$ and observe that for every $(x,y)\in X\times Y$ we get $g_n(x,y)=f_n^{-1}(x,y)\cdot f_{n+1}(x,y)\subset
W_{n}^{-1}\cdot f(x,y)^{-1}\cdot f(x,y)\cdot W_{n+1}=W_{n}^{-1}W_{n+1}\subset \overline{W_n^{-1}W_{n+1}W_{n+2}}\subset W_n^{-1}W_{n+1}W_{n+2}W_{n+3}\subset W_{n-1}$.

The continuity of the group operations on $Z$ implies that the set  $\bar C^s(X\times Y,Z)$ is a group. So, $g_n\in \bar C^s(X\times Y,Z)$ and we can choose a sequence $(g_{n,m})_{m\in\w}$ of continuous functions $g_{n,m}:X\times Y\to G$, which converges to the function $g_n$ in the space $S(X\times Y,Z)$. The functions $g_{n,m}$ can be modified so that $g_{n,m}(X\times Y)\subset W_{n-1}$. Since $W_{-1}=W_0=Z$, this is trivial for $n\le 1$. Put $\bar g_{k,m}=g_{k,m}$ for $k\le 1$, $m\in\w$. Fix any number $n>1$. By the continuity of the function $g_{n,m}$, the set $F_{n,m}=g_{n,m}^{-1}(\overline{W_n^{-1}W_{n+1}W_{n+2}})$ is closed in  $X\times Y$. Since $W_{n-1}$ is an absolute extensor for the space $X\times Y$, the function $g_{n,m}|F_{n,m}$ can be extended to a continuous function $\bar g_{n,m}:X\times Y\to W_{n-1}$. We claim that the function sequence $(\bar g_{n,m})_{m\in\w}$ converges to $g_n$. It suffices to check that for any sub-basic neighborhood $O(g_n)=[K_X\times K_Y,W]$ of the function $g_n$ there exists a number $k\in\w$ such that $\bar g_{n,m}\in O(g_n)$ for all $m\ge k$. Here $W\subset Z$ is an open set and $K_X\subset X$, $K_Y\subset Y$ are compact subsets one of which is a singleton. Using the compactness of the set $g_n(K_X\times K_Y)\subset W$, find a number $l\ge n+2$ such that $g_n(K_X\times K_Y)\cdot W_l\subset W$. The convergence of the function sequence $(g_{n,m})_{m\in\w}$ to $g_n$ yields a number $k\in\w$ such that $g_{n,m}(K_X\times K_Y)\subset  g_n(K_X\times K_Y)\cdot W_l\subset W_n^{-1}W_{n+1}W_l\subset \overline{W_n^{-1}W_{n+1}W_{n+2}}$ for all $m\ge k$. For such numbers $m$ we get the inclusion $K_X\times K_Y\subset F_{n,m}$ implying $\bar g_{n,m}(K_X\times K_Y)=g_{n,m}(K_X\times K_Y)\subset g_n(K_X\times K_Y)\cdot W_l\subset W$. Then $\bar g_{n,m}\in[K_X\times K_Y,W]$, which means that the sequence $(\bar g_{n,m})_{m\in\w}$ converges to $g_n$ in the space $S(X\times Y,Z)$.

For any $n,m\in\w$ consider the continuous function $f_{n,m}=g_{0,m}\cdots g_{n,m}:X\times Y\to Z$ and observe that $\lim_{m\to\infty}f_{n,m}=g_0\cdots g_n=f_{n+1}$ for all $n\in\w$. Repeating the argument of the proof of Theorem~\ref{t2}, we can show that the function sequence $(f_{n,n})_{n\in\w}$ converges to $f$ in the space $S(X\times Y,Z)$. Then $f\in \bar C^s(X\times Y,Z)$ and hence the set $\bar C^s(X\times Y,Z)$ is $\tau_l$-closed in $S(X\times Y,Z)$. By analogy we can prove that this set is $\tau_r$-closed.
\smallskip

The $(\tau_l\cap\tau_r)$-closedness of the set $\bar C^s(X\times Y,Z)$ and the inclusion $S_0(X\times Y,Z)\subset \bar C^s(X\times Y,Z)$ (proved in Theorem~\ref{t2}) implies the inclusion $\bar S^u_0(X\times Y,Z)\subset \bar C^s(X\times Y,Z)$.
\end{proof}

\begin{remark} For the interval $X=Y=[0,1]$ the closedness of the set $\bar C^s(X\times Y,\IR)$ in the topology of uniform convergence on $S(X\times Y,\IR)$ was proved in \cite{VMM1}.
\end{remark}

Theorem~\ref{t3} has two corollaries.

\begin{corollary}\label{c1} For zero-dimensional separable metrizable spaces $X,Y$ and any metrizable topological group $Z$ the set $\bar C^s(X\times Y,Z)$ is $(\tau_l\cap\tau_r)$-closed in $S(X\times Y,Z)$ and hence $\bar S_0^u(X\times Y,Z)\subset \bar C^s(X\times Y,Z)$.
\end{corollary}

\begin{proof} The zero-dimensionality of the space $X\times Y$ implies that each closed subset of $X\times Y$ is a retract of the space $X\times Y$. So, each topological space is an absolute extensor for the space $X\times Y$, which allows us to apply Theorem~\ref{t3} and complete the proof.
\end{proof}

\begin{corollary} For separable metrizable spaces $X,Y$ and a locally convex linear metric space $Z$ the set $\bar C^s(X\times Y,Z)$ is $(\tau_l\cap\tau_r)$-closed in $S(X\times Y,Z)$ and hence $\bar S_0^u(X\times Y,Z)\subset \bar C^s(X\times Y,Z)$.
\end{corollary}

\begin{proof} By the Dugundji Theorem \cite{Dug}, any convex subset of a locally convex space is an absolute extensor for any metrizable space. Consequently, the locally convex space $Z$ has a neighborhood base of (convex) neighborhoods of zero which are absolute extensors for the space $X\times Y$. This allows us to apply Theorem~\ref{t3} and complete the proof of the corollary.
\end{proof}

Observe that an affirmative answer to the ``zero-dimensional'' part of Problem~\ref{prob1} can be derived from Corollary~\ref{c1} and the affirmative answer to the following problem.

\begin{problem} Let $X,Y$ be zero-dimensional compact spaces. Is the set $S_0(X\times Y,\IR)$ of separately continuous functions $f:X\times Y\to \IR$ with zero-dimensional image $f(X\times Y)$ dense in the space $S(X\times Y,\IR)$ endowed with the topology of uniform convergence $\tau_l=\tau_r$?
\end{problem}

This problem can be equivalently reformulated in more elementary terms.

\begin{problem} Let $f:X\times Y\to \IR$ be a separately continuous function defined on the product of two Cantor sets $X=Y=\{0,1\}^\w$. Is there a separately continuous function $g:X\times Y\to\IR$ with (countable) zero-dimensional image $g(X\times Y)$ such that $\sup_{(x,y)\in X\times Y}|f(x,y)-g(x,y)|\le 1$?
\end{problem}

\smallskip

\end{document}